\documentclass[12pt,thmsa,a4paper]{article}
\usepackage{amssymb}
\usepackage{amsfonts}
\usepackage{amsmath}
\usepackage[colorlinks=true]{hyperref}
\usepackage{float}

\allowdisplaybreaks
\hypersetup{urlcolor=blue, citecolor=blue}
\newtheorem{theorem}{Theorem}

\newtheorem{corollary}{Corollary}

\newtheorem{definition}{Definition}
\newtheorem{example}{Example}

\newtheorem{remark}{Remark}

\newenvironment{proof}[1][Proof]{\noindent\textbf{#1.} }{\ \rule{0.5em}{0.5em}}

\begin{document}

\title{A Congruence for Sums of Integer Powers Modulo Products of Distinct
Primes\thanks{%
2010 Mathematics Subject Classi cation: 11A07, 11A25.}\thanks{%
Keywords: Fermat's Little Theorem, number theory, remainder, prime number}}
\author{Shao-Yuan Huang\thanks{%
Department of Mathematics and Information Education, National Taipei
University of Education Taipei 106, Taiwan. Corresponding Author. Email:
syhuang@mail.ntue.edu.tw} \ , \ Hsiu-Yu Wu\thanks{%
Department of Mathematics and Information Education, National Taipei
University of Education Taipei 106, Taiwan}}
\date{}
\maketitle

\begin{abstract}
Let $p_{1},p_{2},...,p_{n}$ be distinct primes and $N_{n}=p_{1}p_{2}\cdots
p_{n}$. We prove that, for any $L>0$ divisible by $\mathrm{lcm}\left(
p_{1}-1,p_{2}-1,...,p_{n}-1\right) $ and $a_{i}\in \mathbb{N}$ satisfying $%
\gcd (a_{i},N_{n})=p_{i}$ for $1\leq i\leq n$, the following congruence
holds:%
\begin{equation*}
a_{1}^{L}+a_{2}^{L}+\cdots +a_{n}^{L}\equiv n-1\ \ \left( \mathrm{mod}%
N_{n}=p_{1}p_{2}\cdots p_{n}\right) .
\end{equation*}%
Furthermore, we formulate and solve remainder problems for $a^{\eta }$ $%
\left( \mathrm{mod}N_{n}\right) $ when $a,\eta \in \mathbb{N}$ and $n=2,3$.
Our approach, based on Fermat's little theorem, unifies and extends several
classical remainder formulas.
\end{abstract}

\section{Introduction}

Remainder problems in number theory are a long-standing essential topic,
particularly in cryptography, primality testing, and computational number
theory. A fundamental challenge is to determine the remainder:%
\begin{equation}
a^{\eta }\equiv ?\text{ \ }\left( \mathrm{mod}p_{1}p_{2}\cdots p_{n}\right) ,
\label{Q1}
\end{equation}%
where $p_{1},p_{2},...,p_{n}$ are distinct primes and $a,\eta \in \mathbb{N}$%
. These problems arise naturally in modular arithmetic and have numerous
applications in modern mathematics and computer science \cite%
{R1,R2,R4,R5,R6,R8,R9,R10,R11,R12,R13}. One classic tool is Fermat's Little
Theorem, which provides a simple method for handling exponentiation when the
modulus is a prime number and plays a crucial role in many number theory
problems. To facilitate the subsequent derivations, we first review the
statement of Fermat's Little Theorem as follows.

\begin{theorem}[Fermat's little theorem]
\label{FLT}Let $p$ be a prime and $a\in \mathbb{N}$. If $\gcd (a,p)=1$, then
$a^{p-1}\equiv 1$ $\left( \mathrm{mod}p\right) $.
\end{theorem}

While Fermat's little theorem is widely used when the modulus is a prime, it
does not directly apply to composite moduli. In such cases, Euler's theorem
provides a more general result by introducing the totient function, allowing
us to extend modular exponentiation techniques to composite numbers.

\begin{theorem}[Euler's Theorem]
\label{ET}Let $a,n\in \mathbb{N}$ with $\gcd (a,n)=1$. Then%
\begin{equation*}
a^{\phi (n)}\equiv 1\text{ \ }\left( \mathrm{mod}n\right)
\end{equation*}%
where $\phi (n)$ is the Euler's totient function, i.e.,%
\begin{equation*}
\phi (n)=\text{count of integers }1\leq x<n\text{ such that }\gcd \left(
x,n\right) =1.
\end{equation*}
\end{theorem}

For a remainder problem, whether one applies Fermat's little theorem or
Euler's theorem, it is often necessary to use the Chinese Remainder Theorem
to handle composite moduli more efficiently \cite{R2,R13}. However, this
process can be rather complex and challenging. To address this, we aim to
develop a set of formulas for the remainder problem (\ref{Q1}), so that more
people can directly use these formulas to obtain the corresponding
remainders. Meanwhile, we discovered an interesting congruence property.
Here, we present a simplified version of this result; its full statement
will be given in Theorem \ref{T3}.%
\begin{equation}
p_{1}^{L}+p_{2}^{L}+\cdots +p_{n}^{L}\equiv n-1\ \ \left( \mathrm{mod}%
p_{1}p_{2}\cdots p_{n}\right)  \label{Q2}
\end{equation}%
where $p_{1},p_{2},...,p_{n}$ are distinct primes and $L>0$ such that $\mathrm{%
lcm}\left( p_{1}-1,p_{2}-1,...,p_{n}-1\right) $ divides $L$.

Next, we give two examples to demonstrate this congruence property (\ref{Q2}%
).

\begin{example}
\label{IE1}Let $a\in \mathbb{N}$. Since $133=7\times 19$, and by Theorem \ref%
{FLT}, we obtain%
\begin{equation}
a^{18}\equiv \left\{
\begin{array}{lll}
0 & \left( \mathrm{mod}133\right) & \text{if }\gcd (a,133)=133,\medskip \\
1 & \left( \mathrm{mod}133\right) & \text{if }\gcd (a,133)=1,\medskip \\
77 & \left( \mathrm{mod}133\right) & \text{if }\gcd (a,133)=7,\medskip \\
57 & \left( \mathrm{mod}133\right) & \text{if }\gcd (a,133)=19.%
\end{array}%
\right.  \label{4b}
\end{equation}%
By (\ref{4b}), we observe that%
\begin{equation}
7^{18}+19^{18}\equiv 77+57\equiv 134\equiv 1\text{ \ }\left( \mathrm{mod}%
133\right) ,  \label{4d}
\end{equation}%
which implies that the property (\ref{Q2}) holds.
\end{example}

\begin{example}
\label{IE2}Let $a\in \mathbb{N}$. Since $66=2\times 3\times 11$, and by
Theorem \ref{FLT}, we obtain%
\begin{equation}
a^{10}\equiv \left\{
\begin{array}{lll}
1 & \left( \mathrm{mod}66\right) & \text{if }\gcd (a,66)=1,\medskip \\
55 & \left( \mathrm{mod}66\right) & \text{if }\gcd (a,66)=11,\medskip \\
45 & \left( \mathrm{mod}66\right) & \text{if }\gcd (a,66)=3,\medskip \\
34 & \left( \mathrm{mod}66\right) & \text{if }\gcd (a,66)=2.%
\end{array}%
\right.  \label{4c}
\end{equation}%
By (\ref{4c}), we observe that%
\begin{equation}
2^{10}+3^{10}+11^{10}\equiv 34+45+55\equiv 134\equiv 2\text{ \ }\left( \mathrm{%
mod}66\right) ,  \label{4e}
\end{equation}%
which implies that the property (\ref{Q2}) holds.
\end{example}

\section{Main Results}

In this section, we present our main results. We begin with the following
definition.

\begin{definition}
\label{DE}Let $p$ be a prime and $m\in \mathbb{N}$. The number $m_{p}\in
\{1,2,...,p-1\}$ denotes the multiplicative inverse of $m$ in $Z_{p}$, that
is,
\begin{equation*}
m\times m_{p}\equiv 1\text{ \ }\left( \mathrm{mod}p\right) .
\end{equation*}%
Note that such a number $m_{p}$ exists and unieque because $Z_{p}$ is a
field.
\end{definition}

\smallskip

The following Theorem \ref{T1} provides a formulation of the remainder
problem (\ref{Q1}) in the case $n=2$.

\begin{theorem}
\label{T1}Let $p$ and $q$ be distinct primes, and let $a\in \mathbb{N}$.
Then the following statements (i)--(iv) hold.

\begin{itemize}
\item[(i)] If $\gcd (a,pq)=pq$, then $a^{\eta }\equiv 0\ \left( \mathrm{mod}%
pq\right) .$

\item[(ii)] If $\gcd (a,pq)=1$, then $a^{\mathrm{lcm}(p-1,q-1)\eta }\equiv 1\
\left( \mathrm{mod}pq\right) $.

\item[(iii)] If $\gcd (a,pq)=q$, then $a^{\left( p-1\right) \eta }\equiv
qq_{p}\ \left( \mathrm{mod}pq\right) $.

\item[(iv)] If $\gcd (a,pq)=p$, then $a^{\left( q-1\right) \eta }\equiv
1-qq_{p}\ \left( \mathrm{mod}pq\right) $. Moreover, $pq+1-qq_{p}$ is the
remainder when $a^{\left( q-1\right) \eta }$ is divided by $pq$.
\end{itemize}
\end{theorem}

\begin{remark}
Recall Euler's totient function $\phi (n)$ defined by Theorem \ref{ET}. Note
that $\phi (pq)=\left( p-1\right) \left( q-1\right) $ when $p$ and $q$ are
distinct primes. Hence, Theorem \ref{T1}(ii) is exactly Euler's Theorem for
the case $n=pq.$
\end{remark}

By Theorem \ref{T1}, we obtain the following corollaries: Corollaries \ref%
{C2} and \ref{C1}.

\begin{corollary}
\label{C2}Let $p$ and $q$ be distinct primes, and let $a_{1},a_{2},\eta
_{1},\eta _{2}\in \mathbb{N}$. Assume that $\gcd (a_{1},pq)=p$ and $\gcd
(a_{2},pq)=q$. Then%
\begin{equation*}
a_{1}^{\left( q-1\right) \eta _{1}}+a_{2}^{\left( p-1\right) \eta
_{2}}\equiv 1\text{ \ }\left( \mathrm{mod}pq\right) .
\end{equation*}%
In particular, $a_{1}^{L}+a_{2}^{L}\equiv 1$ $\left( \mathrm{mod}pq\right) $
for any $L>0$ such that $\mathrm{lcm}(p-1,q-1)$ divides $L$.
\end{corollary}

\begin{corollary}
\label{C1}Let $p>2$ be a prime and $a,\eta \in \mathbb{N}$. Then the
following statements (i)--(iii) hold.

\begin{itemize}
\item[(i)] If $\gcd (a,2p)=1$, then $a^{(p-1)\eta }\equiv 1\ \left( \mathrm{mod%
}2p\right) $.

\item[(ii)] If $\gcd (a,2p)=q$, then $a^{\eta }\equiv p\ \left( \mathrm{mod}%
2p\right) $. Moreover, $p^{\eta }\equiv p$ $(\mathrm{mod}2p)$, as illustrated
in Table \ref{Tab1}.

\item[(iii)] If $\gcd (a,2p)=2$, then $a^{\left( p-1\right) \eta }\equiv
1+p\ \left( \mathrm{mod}2p\right) $. Moreover, $2^{16\eta }\equiv 1+p$ $(\mathrm{%
mod}2p)$.
\end{itemize}
\end{corollary}

\begin{table}[tbph]
\renewcommand{\arraystretch}{1.2} % 將此數值調大即可增加行高
\centering
\begin{tabular}{|c|c|c|}
\hline
$3^\eta \equiv 3 \pmod{6}$ & $5^\eta \equiv 5 \pmod{10}$ & $7^\eta \equiv 7 %
\pmod{14}$ \\ \hline
$11^\eta \equiv 11 \pmod{26}$ & $13^\eta \equiv 13 \pmod{26}$ & $17^\eta
\equiv 17 \pmod{28}$ \\ \hline
\vdots & \vdots & \vdots \\ \hline
\end{tabular}%
\caption{$p^{\protect\eta }\equiv p\pmod{2q}$ for $\protect\eta \in \mathbb{N%
}$ where $p$ is a prime.}
\label{Tab1}
\end{table}

The following Theorem \ref{T3} is the most important result of this paper.

\begin{theorem}
\label{T3}Let $p_{1},p_{2},...,p_{n}$ be distinct primes, and let $%
a_{1},a_{2},...,a_{n}\in \mathbb{N}$. Assume that $\gcd
(a_{i},p_{1}p_{2}\cdots p_{n})=p_{i}$ for $1\leq i\leq n$. Then%
\begin{equation*}
a_{1}^{L}+a_{2}^{L}+\cdots +a_{n}^{L}\equiv n-1\ \ \left( \mathrm{mod}%
p_{1}p_{2}\cdots p_{n}\right) ,
\end{equation*}%
for any $L>0$ such that $\mathrm{lcm}\left( p_{1}-1,p_{2}-1,...,p_{n}-1\right)
$ divides $L$.
\end{theorem}

\begin{remark}
As $a_{i}=p_{i}$ for $1\leq i\leq n$, Theorem \ref{T3} varifies the
interesting congruence property (\ref{Q2}). In addition, as $n=2$, Corollary %
\ref{C2} is more general than Theorem \ref{T3}. For instance, we consider $%
p=3$ and $q=5$. By Corollary \ref{C2}, we obtain%
\begin{equation*}
3^{4\eta _{1}}+5^{2\eta _{2}}\equiv 1\text{ \ }\left( \mathrm{mod}15\right)
\end{equation*}%
for any $\eta _{1},\eta _{2}\in \mathbb{N}$. However, by Theorem \ref{T3},
we obtain $3^{8}+5^{8}\equiv 1$ $\left( \mathrm{mod}15\right) $.
\end{remark}

Finally, we apply similar techniques used in the proof of Theorem \ref{T1}
to obtain the formulas for the remainder problem (\ref{Q1}) with $n=3$.

\begin{theorem}
\label{T2}Let $p<q<r$ be primes, $a\in \mathbb{N}$, and $L=\mathrm{lcm}\left(
p-1,q-1,r-1\right) $. Assume that
\begin{equation}
qr\equiv 1\ \left( \mathrm{mod}p\right) .  \label{H1}
\end{equation}%
Then the following statements (i)--(v) hold.

\begin{itemize}
\item[(i)] If $\gcd (a,pqr)=1$, then $a^{L}\equiv 1\ \left( \mathrm{mod}%
pqr\right) $.

\item[(ii)] If $\gcd (a,pqr)=r$, then $a^{L}\equiv rq\left( 1-aa_{q}\right)
+aa_{q}$ $\left( \mathrm{mod}pqr\right) .$

\item[(iii)] If $\gcd (a,pqr)=q$, then $a^{L}\equiv r\left( 1-qr\right) \bar{%
\beta}+1\ \left( \mathrm{mod}pqr\right) $ where $\bar{\beta}\in
\{0,1,...,q-1\} $ satisfies that $r\bar{\beta}+1\equiv 0$ $\left( \mathrm{mod}%
q\right) $.

\item[(iv)] If $\gcd (a,pqr)=p$, then $a^{L}\equiv 1-rq\ \left( \mathrm{mod}%
pqr\right) .$

\item[(v)] $a_{1}^{L}+a_{2}^{L}+a_{3}^{L}\equiv 2$ $\left( \mathrm{mod}%
pqr\right) $ where $a_{1},a_{2},a_{3}\in \mathbb{N}$ satisfy that
\begin{equation*}
\gcd (a_{1},pqr)=p\text{, \ }\gcd (a_{2},pqr)=q\text{ \ and \ }\gcd
(a_{3},pqr)=r.
\end{equation*}
\end{itemize}
\end{theorem}

\begin{remark}
From the proof of Theorem \ref{T2}, we see that Theorem \ref{T2}(v) follows
from Theorem \ref{T2}(ii)--(iv). This conclusion is also consistent with
Theorem \ref{T3} in the case $n=3.$
\end{remark}

\section{Proofs of Main Results}

In this section, we present the proofs of theorems.

\smallskip

\begin{proof}[Proof of Theorem \protect\ref{T1}]
(I) Assume that $\gcd (a,pq)=pq$. Then $a\equiv 0$ $\left( \mathrm{mod}%
pq\right) $, which implies that $a^{\eta }\equiv 0$ $\left( \mathrm{mod}%
pq\right) $. Thus, statement (i) holds.

(II) Assume that $\gcd (a,pq)=1$. Thus, we have $\gcd (a,p)=\gcd (a,q)=1$.
By Theorem \ref{FLT}, we obtain%
\begin{equation}
a^{p-1}\equiv 1\text{ }\left( \mathrm{mod}p\right) \text{ \ and \ }%
a^{q-1}\equiv 1\text{ }\left( \mathrm{mod}q\right) .  \label{c2}
\end{equation}%
In addition, there exist $\alpha _{1},\alpha _{2}\in \mathbb{N}$ such that
\begin{equation}
\mathrm{lcm}(p-1,q-1)=\left( p-1\right) \alpha _{1}\text{ \ and \ }\mathrm{lcm}%
(p-1,q-1)=\left( q-1\right) \alpha _{2}.  \label{c2b}
\end{equation}%
By (\ref{c2}) and (\ref{c2b}), then%
\begin{equation*}
\left\{
\begin{array}{l}
a^{\mathrm{lcm}(p-1,q-1)\eta }=\left( a^{p-1}\right) ^{\alpha _{1}\eta }\equiv
1\text{ \ }\left( \mathrm{mod}p\right) ,\medskip \\
a^{\mathrm{lcm}(p-1,q-1)\eta }=\left( a^{q-1}\right) ^{\alpha _{2}\eta }\equiv
1\text{ \ }\left( \mathrm{mod}q\right) .%
\end{array}%
\right.
\end{equation*}%
Since $\gcd (p,q)=1$, we see that $a^{\mathrm{lcm}(p-1,q-1)\eta }\equiv 1$ $%
\left( \mathrm{mod}pq\right) $. Thus, statement (ii) holds.

(III) Assume that $\gcd (a,pq)=q$. Thus, we have $\gcd (a,p)=1$ and $a\equiv
0$ $\left( \mathrm{mod}q\right) $. By Theorem \ref{FLT}, we obtain%
\begin{equation}
a^{\left( p-1\right) \eta }\equiv 1\text{ }\left( \mathrm{mod}p\right) \text{
\ and \ }a^{\left( p-1\right) \eta }\equiv 0\text{ \ }\left( \mathrm{mod}%
q\right) .  \label{c3c}
\end{equation}%
Since $\gcd (a,pq)=q$, there exists $k_{1}\in \mathbb{N}$ such that%
\begin{equation*}
a^{\left( p-1\right) \eta }=qk_{1}\text{ \ and \ }\gcd (k_{1},p)=1.
\end{equation*}%
Then there exist $\beta _{1}\in \mathbb{Z}$ and $r_{1}\in \{1,2,...,p-1\}$
such that $k_{1}=p\beta _{1}+r_{1}$. It follows that
\begin{equation}
a^{\left( p-1\right) \eta }=qk_{1}=q\left( p\beta _{1}+r_{1}\right) =pq\beta
_{1}+qr_{1}\text{.}  \label{c3}
\end{equation}%
By (\ref{c3c}) and (\ref{c3}), then
\begin{equation}
q\bar{q}\equiv 1\equiv a^{\left( p-1\right) \eta }\equiv qr_{1}\text{ \ }%
\left( \mathrm{mod}p\right) .  \label{c3d}
\end{equation}%
Since $\gcd (p,q)=1$, and by (\ref{c3d}), we see that $q_{p}\equiv r_{1}$ $%
\left( \mathrm{mod}p\right) $. So $q_{p}=r_{1}$. Then by (\ref{c3}), we see
that $a^{\eta }\equiv qq_{p}$ $\left( \mathrm{mod}pq\right) $. Thus, statement
(iii) holds.

(IV) Assume that $\gcd (a,pq)=p$. Thus, we have $\gcd (a,q)=1$ and $a\equiv
0 $ $\left( \mathrm{mod}p\right) $. By Theorem \ref{FLT}, we obtain%
\begin{equation}
a^{\left( q-1\right) \eta }\equiv 0\text{ }\left( \mathrm{mod}p\right) \text{
\ and \ }a^{\left( q-1\right) \eta }\equiv 1\text{ }\left( \mathrm{mod}%
q\right) ,  \label{c4a}
\end{equation}%
from which it follows that $a^{\left( q-1\right) \eta }=qk_{2}+1$ for some $%
k_{2}\in \mathbb{Z}$ . Then there exist $\beta _{2}\in \mathbb{Z}$ and $%
r_{2}\in \{0,1,...,p-1\}$ such that $k_{2}=p\beta _{2}+r_{2}$. It follows
that%
\begin{equation}
a^{\left( q-1\right) \eta }=qk_{2}+1=q\left( p\beta _{2}+r_{2}\right)
+1=pq\beta _{2}+qr_{2}+1\text{.}  \label{c4c}
\end{equation}%
By (\ref{c4a}) and (\ref{c4c}), we observe that%
\begin{equation*}
1-qq_{p}\equiv 0\equiv a^{\left( q-1\right) \eta }\equiv qr_{2}+1\text{ \ }%
\left( \mathrm{mod}p\right) ,
\end{equation*}%
which implies that%
\begin{equation}
q\left( r_{2}+q_{p}\right) \equiv 0\text{ \ }\left( \mathrm{mod}p\right) ,
\label{c4d}
\end{equation}%
Since $\gcd (p,q)=1$, and by (\ref{c4d}), we see that $-q_{p}\equiv r_{2}$ $%
\left( \mathrm{mod}p\right) $. So $p-q_{p}=r_{2}$. Then by (\ref{c4c}), we see
that%
\begin{equation}
a^{\left( q-1\right) \eta }=pq\beta _{2}+q\left( p-q_{p}\right) +1\equiv
1-qq_{p}\text{ \ }\left( \mathrm{mod}pq\right) .  \label{c4e}
\end{equation}%
Since $1<qq_{p}<pq$, and by (\ref{c4e}), we see that $pq+1-qq_{p}$ is the
remainder when $a^{\left( q-1\right) \eta }$ is divided by $pq$. Thus,
statement (iv) holds. The proof is complete.
\end{proof}

\smallskip

\begin{proof}[Proof of Corollary \protect\ref{C2}]
By Theorem \ref{T1}(iii)(iv), we obtain%
\begin{equation}
a_{1}^{\left( q-1\right) \eta _{1}}+a_{2}^{\left( p-1\right) \eta
_{2}}\equiv qq_{p}+\left( 1-qq_{p}\right) \equiv 1\text{ \ }\left( \mathrm{mod}%
pq\right) .  \label{C2a}
\end{equation}%
Since $\mathrm{lcm}(p-1,q-1)$ $|$ $L$, there exists $\alpha _{1},\alpha
_{2}\in \mathbb{N}$ such that $L=\left( q-1\right) \alpha _{1}=\left(
p-1\right) \alpha _{2}$. So by (\ref{C2a}),%
\begin{equation*}
a_{1}^{L}+a_{2}^{L}\equiv a_{1}^{\left( q-1\right) \alpha
_{1}}+a_{2}^{\left( p-1\right) \alpha _{2}}\equiv 1\text{ \ }\left( \mathrm{mod%
}pq\right) .
\end{equation*}%
The proof is complete.
\end{proof}

\smallskip

\begin{proof}[Proof of Corollary \protect\ref{C1}]
Let $q=2$. Since $p$ is odd, we see that $p_{q}=1$. So Corollary \ref{C1}
follows from Theorem \ref{T1}. The proof is complete.
\end{proof}

\smallskip

\begin{proof}[Proof of Theorem \protect\ref{T3}]
Theorem \ref{T3} immediately follows as $n=1$. By Corollary \ref{C2}, then
Theorem \ref{T3} holds for $n=2$. Assume that Theorem \ref{T3} holds for $%
n=k $.

Next, we consider $n=k+1$. Let $L>0$ such that $\mathrm{lcm}\left(
p_{1}-1,p_{2}-1,...,p_{k+1}-1\right) $ divides $L$. It implies that $\mathrm{%
lcm}\left( p_{1}-1,p_{2}-1,...,p_{k}-1\right) $ divides $L$. Then
\begin{equation}
a_{1}^{L}+a_{2}^{L}+\cdots +a_{k}^{L}\equiv k-1\ \ \left( \mathrm{mod}%
p_{1}p_{2}\cdots p_{k}\right) .  \label{te7}
\end{equation}%
In addition, since $\gcd (a_{k+1},p_{i})=1$ for $i=1,2,...,k$, and by
Theorem \ref{FLT}, we observe that
\begin{equation*}
a_{k+1}^{L}\equiv 1\ \ \left( \mathrm{mod}p_{i}\right) \text{ \ for }%
i=1,2,...,k.
\end{equation*}%
Since $p_{1},p_{2},...,p_{k}$ are distinct primes, we see that
\begin{equation}
a_{k+1}^{L}\equiv 1\ \ \left( \mathrm{mod}p_{1}p_{2}\cdots p_{k}\right) .
\label{te6}
\end{equation}%
By (\ref{te7}) and (\ref{te6}), then%
\begin{equation}
a_{1}^{L}+a_{2}^{L}+\cdots +a_{k+1}^{L}\equiv k\text{ \ }\left( \mathrm{mod}%
p_{1}\cdots p_{k}\right) .  \label{te4}
\end{equation}%
Since $\gcd (a_{i},p_{k+1})=1$ and $\gcd (a_{k+1},p_{k+1})=p_{k+1}$ for $%
i=1,2,...,k$, and by Fermat's Little Theorem, we observe that%
\begin{equation*}
a_{i}^{L}\equiv 1\text{ \ }\left( \mathrm{mod}p_{k+1}\right) \text{ \ and \ }%
a_{k+1}^{L}\equiv 0\text{ \ }\left( \mathrm{mod}p_{k+1}\right) \text{ \ for }%
i=1,2,...,k,
\end{equation*}%
which implies that
\begin{equation}
a_{1}^{L}+a_{2}^{L}+\cdots +a_{k+1}^{L}\equiv k\text{ \ }\left( \mathrm{mod}%
p_{k+1}\right) .  \label{te5}
\end{equation}%
By (\ref{te4}) and (\ref{te5}), we obtain%
\begin{equation*}
p_{1}^{L}+p_{2}^{L}+\cdots +p_{k+1}^{L}\equiv k\text{ \ }\left( \mathrm{mod}%
p_{1}\cdots p_{k}p_{k+1}\right) ,
\end{equation*}%
which implies that Theorem \ref{T3} holds for $n=k+1$. So Theorem \ref{T3}
holds by mathematical induction. The proof is complete.
\end{proof}

\smallskip

\begin{proof}[Proof of Theorem \protect\ref{T2}]
(I) Assume that $\gcd (a,pqr)=1.$ Since $L=\mathrm{lcm}\left(
p-1,q-1,r-1\right) $, and by Theorem \ref{FLT}, we observe that
\begin{equation*}
a^{L}\equiv 1\ \ \left( \mathrm{mod}p\right) ,\text{ \ }a^{L}\equiv 1\ \
\left( \mathrm{mod}q\right) \text{ \ and \ }a^{L}\equiv 1\ \ \left( \mathrm{mod}%
r\right) .
\end{equation*}%
Since $p$, $q$ and $r$ are primes, it follows that $a^{L}\equiv 1\ \left(
\mathrm{mod}pqr\right) $. Thus, statement (i) holds.

(II) Assume that $\gcd (a,pqr)=r$. It follows that $a^{L}=a\left( q\alpha
+\beta \right) $ for some $\alpha \in \mathbb{Z}$ and $\beta \in
\{1,2,...,q-1\}$. Since $\gcd (a,q)=1$, and by Theorem \ref{FLT}, we see
that
\begin{equation*}
a^{L}=a\left( q\alpha +\beta \right) \equiv a\beta \equiv 1\text{ \ }\left(
\mathrm{mod}q\right) ,
\end{equation*}%
which implies that $\beta =a_{q}$. Since $\gcd (a,pqr)=r$, we see that $a=rw$
for some $w\in \mathbb{N}$. By Theorem \ref{FLT} and (\ref{H1}), we obtain%
\begin{equation*}
a^{L}=a\left( q\alpha +a_{q}\right) =\left( rq\right) w\alpha +aa_{q}\equiv
w\alpha +aa_{q}\equiv 1\text{\ \ }\left( \mathrm{mod}p\right) ,
\end{equation*}%
which implies that $w\alpha =1-aa_{q}+pK_{1}$ for some $K\in \mathbb{Z}$.
Therefore, we observe that%
\begin{eqnarray*}
a^{L} &=&a\left( q\alpha +a_{q}\right) \\
&=&rq\left( w\alpha \right) +aa_{q} \\
&=&rq\left( 1-aa_{q}+pK\right) +aa_{q} \\
&\equiv &rq\left( 1-aa_{q}\right) +aa_{q}\text{ \ }\left( \mathrm{mod}%
pqr\right) .
\end{eqnarray*}%
Thus, statement (ii) holds.

(III) Assume that $\gcd (a,pqr)=q$. Since $\gcd (a,r)=1$, and by Theorem \ref%
{FLT}, we see that $a^{L}\equiv 1$ $\left( \mathrm{mod}r\right) $. It follows
that
\begin{equation}
a^{L}=r\alpha _{1}+1=r\left( q\alpha _{2}+\bar{\beta}\right) +1  \label{z2}
\end{equation}%
for some $\alpha _{1},\alpha _{2}\in \mathbb{Z}$ and $\bar{\beta}\in
\{0,1,...,q-1\}$. Since $a\equiv 0$ $\left( \mathrm{mod}q\right) $, and by (%
\ref{z2}), we see that $r\bar{\beta}+1\equiv 0$ $\left( \mathrm{mod}q\right) $%
. Since $\gcd (a,p)=1$, and by Theorem \ref{FLT}, (\ref{H1}) and (\ref{z2}),
we obtain%
\begin{equation*}
a^{L}\equiv \alpha _{2}+r\bar{\beta}+1\equiv 1\text{ \ }\left( \mathrm{mod}%
p\right) ,
\end{equation*}%
which implies that $\alpha _{2}+r\bar{\beta}\equiv 0$ $\left( \mathrm{mod}%
p\right) $. Then $\alpha _{2}=p\bar{K}-r\bar{\beta}$ for some $\bar{K}\in
\mathbb{Z}$. So by (\ref{z2}), we observe that%
\begin{eqnarray*}
a^{L} &=&r\left[ q\left( p\bar{K}-r\bar{\beta}\right) +\bar{\beta}\right] +1
\\
&=&pqr\bar{K}+r\left( 1-qr\right) \bar{\beta}+1 \\
&\equiv &r\left( 1-qr\right) \bar{\beta}+1\text{ \ }\left( \mathrm{mod}%
pqr\right) .
\end{eqnarray*}%
Thus, statement (iii) holds.

(IV) Assume that $\gcd (a,pqr)=p$. Since $\gcd (a,r)=1$, we have (\ref{z2}).
Since $\gcd (a,q)=1$, and by Theorem \ref{FLT} and (\ref{z2}), we see that%
\begin{equation*}
a^{L}\equiv r\bar{\beta}+1\equiv 1\text{ \ }\left( \mathrm{mod}q\right) ,
\end{equation*}%
which implies that $r\bar{\beta}\equiv 0$ $\left( \mathrm{mod}q\right) $.
Since $\gcd (r,q)=1$, we see that $\bar{\beta}=0$. Since $a\equiv 0$ $\left(
\mathrm{mod}p\right) $, and by (\ref{H1}) and (\ref{z2}), we observe that%
\begin{equation*}
a^{L}=rq\alpha _{2}+1\equiv \alpha _{2}+1\equiv 0\text{ \ }\left( \mathrm{mod}%
p\right) ,
\end{equation*}%
which implies that $\alpha _{2}=p\tilde{K}-1$ for some $\tilde{K}\in \mathbb{%
Z}$. Then
\begin{equation*}
a^{L}=rq\alpha _{2}+1=rq\left( p\tilde{K}-1\right) +1\equiv -rq+1\text{ \ }%
\left( \mathrm{mod}pqr\right) .
\end{equation*}%
Thus, statement (iv) holds.

(V) By Theorem \ref{T2}(ii)--(iv), then%
\begin{eqnarray}
r^{L}+q^{L}+p^{L} &\equiv &\left[ rq\left( 1-rr_{q}\right) +rr_{q}\right] +%
\left[ r\left( 1-qr\right) \bar{\beta}+1\right] +\left( 1-rq\right) \text{ \
}\left( \mathrm{mod}pqr\right) \smallskip  \notag \\
&\equiv &r\left( 1-qr\right) \left( \bar{\beta}+r_{q}\right) +2\text{ \ }%
\left( \mathrm{mod}pqr\right) .  \label{z4}
\end{eqnarray}%
Since $rr_{q}\equiv 1$ $\left( \mathrm{mod}q\right) $ and $r\bar{\beta}%
+1\equiv 0$ $\left( \mathrm{mod}q\right) $, we observe that%
\begin{equation}
r\left( r_{q}+\bar{\beta}\right) =\left( rr_{q}-1\right) +\left( r\bar{\beta}%
+1\right) \equiv 0\text{ \ }\left( \mathrm{mod}q\right) .  \label{z5}
\end{equation}%
Since $\gcd (q,r)=1$, and by (\ref{z5}), we see that
\begin{equation}
\bar{\beta}+r_{q}\equiv 0\ \ \left( \mathrm{mod}q\right) .  \label{z6}
\end{equation}%
By (\ref{H1}), we obtain that $1-qr\equiv 0$ $\left( \mathrm{mod}p\right) $.
So by (\ref{z4}) and (\ref{z6}), $r^{L}+q^{L}+p^{L}\equiv 2$ $\left( \mathrm{%
mod}pqr\right) .$ Thus, statement (v) holds. The proof is complete.
\end{proof}

\section{Examples}

In this section, we provide two examples to illustrate our main results.

\begin{example}
\label{Ex2}Let $a,\eta \in \mathbb{N}$. Since $133=7\times 19$, we consider $%
p=7$ and $q=19$. Since $1\equiv qq_{p}\equiv 5q_{p}\ \left( \mathrm{mod}%
7\right) $, we obtain $q_{p}=3$. So by Theorem \ref{T1}, we obtain%
\begin{equation}
\left\{
\begin{array}{lll}
a^{\eta }\equiv 0 & \left( \mathrm{mod}133\right) & \text{if }\gcd
(a,133)=133,\medskip \\
a^{18\eta }\equiv 1 & \left( \mathrm{mod}133\right) & \text{if }\gcd
(a,133)=1,\medskip \\
a^{18\eta }\equiv 77 & \left( \mathrm{mod}133\right) & \text{if }\gcd
(a,133)=7,\medskip \\
a^{6\eta }\equiv 57 & \left( \mathrm{mod}133\right) & \text{if }\gcd
(a,133)=19.%
\end{array}%
\right.  \label{2a}
\end{equation}%
Futhermore, $19^{6\eta _{1}}+7^{18\eta _{2}}\equiv 1$ $\left( \mathrm{mod}%
133\right) $ for any $\eta _{1},\eta _{2}\in \mathbb{N}$. Clearly, equation (%
\ref{2a}) provides more information about the remainder than equation (\ref%
{4b}).
\end{example}

\bigskip

\begin{example}
\label{Ex4}Let $a\in \mathbb{N}$. Since $66=2\times 3\times 11$, we consider
$p=2$, $q=3$ and $r=11$. Notice that
\begin{equation*}
qr=33\equiv 1\ \left( \mathrm{mod}p=2\right) ,\text{ \ and \ }L=\mathrm{lcm}%
(p-1,q-1,r-1)=10,
\end{equation*}%
which implies that (\ref{H1}) holds. To prove (\ref{4c}), we consider four
cases.

\smallskip \noindent Case 1. Assume that $\gcd (a,66)=1$. By Theorem \ref{T2}%
(i), then $a^{10}\equiv 1$ $\left( \mathrm{mod}66\right) .$

\smallskip \noindent Case 2. Assume that $\gcd (a,66)=11$. There exists $%
k\in \mathbb{Z}$ such that $a=11(6k+1)$ or $11\left( 6k+5\right) .$

If $a=11(6k+1)$, then
\begin{equation*}
aa_{q}=11(6k+1)a_{q}\equiv 2a_{q}\equiv 1\ \ \left( \mathrm{mod}q=3\right) ,
\end{equation*}%
which implies that $a_{q}=2$. So by Theorem \ref{T2}(ii), then%
\begin{eqnarray*}
a^{10} &\equiv &rq\left( 1-aa_{q}\right) +aa_{q} \\
&\equiv &33+2a\text{ \ }\left( \mathrm{mod}66\right) \\
&\equiv &33+22(6k+1)\text{ \ }\left( \mathrm{mod}66\right) \\
&\equiv &55\text{ \ }\left( \mathrm{mod}66\right) .
\end{eqnarray*}%
If $a=11\left( 6k+5\right) $, then
\begin{equation*}
aa_{q}=11(6k+5)a_{q}\equiv a_{q}\equiv 1\ \ \left( \mathrm{mod}3\right) ,
\end{equation*}%
which implies that $a_{q}=1$. By Theorem \ref{T2}(ii), then
\begin{eqnarray*}
a^{10} &\equiv &rq\left( 1-aa_{q}\right) +aa_{q} \\
&\equiv &33-32a\ \ \left( \mathrm{mod}66\right) \\
&\equiv &33-32\times 11\left( 6k+5\right) \ \ \left( \mathrm{mod}66\right) \\
&\equiv &55\text{ \ }\left( \mathrm{mod}66\right) .
\end{eqnarray*}%
Thus $a^{10}\equiv 55$ $\left( \mathrm{mod}66\right) .$

\smallskip \noindent Case 3. Assume that $\gcd (a,66)=3$. Since $11\bar{\beta%
}+1\equiv 0$ $\left( \mathrm{mod}3\right) $, we obtain $\bar{\beta}=1$. So by
Theorem \ref{T2}(iii), then $a^{10}\equiv 45\ \left( \mathrm{mod}66\right) $.

\smallskip \noindent Case 4. Assume that $\gcd (a,66)=2$. By Theorem \ref{T2}%
(iv), then $a^{10}\equiv 34$ $\left( \mathrm{mod}66\right) .$

From the above discussions, we prove (\ref{4c}).
\end{example}

\smallskip

\begin{example}
\label{Ex5}Consider $p_{1}=3$, $p_{2}=7$, $p_{3}=11$ and $p_{4}=17$. It is
easy to compute%
\begin{equation*}
\mathrm{lcm}\left( p_{1}-1,p_{2}-1,p_{3}-1,p_{4}-1\right) =\mathrm{lcm}\left(
2,6,10,16\right) =240.
\end{equation*}
By Theorem \ref{T2}, then%
\begin{equation*}
3^{^{240\eta }}+7^{^{240\eta }}+11^{^{240\eta }}+17^{^{240\eta }}\equiv 4%
\text{ \ }\left( \mathrm{mod}3927=3\times 7\times 11\times 17\right)
\end{equation*}%
for any $\eta \in \mathbb{N}$.
\end{example}

\bigskip

\noindent \textbf{Acknowledge.} We would like to thank Prof. J.-C. Cheng for
his valuable suggestions.

\end{document}